\documentclass[bezier,amstex, oneside,reqno]{amsart}
\usepackage{amsmath, amssymb, amsthm, verbatim, euscript, amscd, textcomp}
\usepackage[dvips]{graphicx}
\usepackage{epsfig}
\usepackage{color}

\def\ie{\emph{i.e., }}
\def\eg{\emph{e.g., }}

\def\R{\mathbb R}
\def\Z{\mathbb Z}
\def\U{\mathcal U}
\def\V{\mathcal V}
\def\W{\mathcal W}
\def\P{\EuScript P}
\def\C{\EuScript C}
\def\T{\mathbb T}
\def\qqZ{\left(\frac1q\Z\right)^n}
\def\qZ{\frac1q\Z^n}
\def\E{\EuScript E}
\def\AA{\EuScript A}
\def\A{\mathbb A}
\def\TT{\mathcal T}

\newtheorem{theorem}{Theorem}[section]
\newtheorem{prop}[theorem]{Proposition}

\newtheorem{claim}[theorem]{Claim}

 \theoremstyle{remark}
\newtheorem{remark}[theorem]{Remark}

\theoremstyle{remark}

\begin{document}
\author{ F. Thomas Farrell  and Andrey Gogolev$^\ast$}
\title[Anosov diffeomorphisms]{Anosov diffeomorphisms constructed
from $\boldsymbol{\pi}_k(\bf\textup{Diff}(S^n))$}
\thanks{$^\ast$Both authors were partially supported by NSF grants.}
\begin{abstract}
We construct Anosov diffeomorphisms on manifolds that are
homeomorphic to infranilmanifolds yet have exotic smooth structures.
These manifolds are obtained from standard infranilmanifolds by
connected summing with certain exotic spheres. Our construction
produces Anosov diffeomorphisms of high codimension on infranilmanifolds
with irreducible exotic smooth structures.
\end{abstract}
\date{}
 \maketitle

\section{Introduction}
Let $M$ be a compact smooth $n$-dimensional Riemannian manifold.
Recall that a diffeomorphism $f$ is called {\it Anosov} if there
exist constants $\lambda \in (0,1)$ and $C>0$ along with a
$df$-invariant splitting $TM=E^s\oplus E^u$ of the tangent bundle of
$M$, such that for all $m \ge 0$
\begin{multline*}
\qquad\|df^mv\|\le C\lambda^m\|v\|,\;v\in E^s,\; \\
\qquad\shoveleft{\|df^{-m}v\|\le C\lambda^{m}\|v\|,\;v\in E^u.
\hfill}
\end{multline*}

If either fiber of $E^s$ or $E^u$ has dimension $k$ with $k\le\lfloor n/2\rfloor$ then $f$ is
called a {\it codimension $k$} Anosov diffeomorphism. See,
\eg~\cite{KH} for background on Anosov diffeomorphisms.

The classification of Anosov diffeomorphisms is an outstanding open
problem. All currently known examples of manifolds that support
Anosov diffeomorphisms are homeomorphic to infranilmanifolds.
It is an interesting question to study existence of Anosov diffeomorphisms on
manifolds that are homeomorphic to infranilmanifolds yet have exotic smooth structure.

Farrell and Jones have constructed~\cite{FJ} codimension one Anosov
diffeomorphisms on higher dimensional exotic tori, \ie manifolds that
are homeomorphic to tori but have non-standard smooth structure. The
current paper should be considered as a sequel to~\cite{FJ}. We formulate our
main result below. (See Section~\ref{sec_backg} for the definitions of the Gromoll groups and expanding endomorphism.)

\begin{theorem}
\label{thm_main}
Let $M$ be an $n$-dimensional ($n\ge7$) infranilmanifold; in particular, $M$
can be a nilmanifold. Let $L\colon M\to M$ be a
codimension $k$ Anosov automorphism. Assume that $M$ admits an
expanding endomorphism $E\colon M\to M$ that commutes with $L$. Let
$\Sigma$ be a homotopy sphere from the Gromoll group
$\Gamma_{k+1}^n$, then the connected sum $M\# \Sigma$ admits a
codimension $k$ Anosov diffeomorphism.
\end{theorem}
\begin{remark}
 Not all exotic smooth structures on infranilmanifolds come from exotic
spheres. For example, obstruction theory gives smooth
structures on tori that are not even PL-equivalent~\cite{HS} (note that by
Alexander trick all exotic tori that come from exotic spheres are
PL-equivalent). In particular, there are three different smooth
structures on $\mathbb T^5$~\cite[p. 227]{W}. It would be very interesting to see if
these manifolds support Anosov diffeomorphisms.
\end{remark}

Our construction of Anosov diffeomophisms on exotic infranilmanifolds is
different from that of Farrell and Jones and gives Anosov diffeomorphisms of high codimension. 
Of course, one can multiply the Anosov diffeomorphism of Farrell and Jones on $\T^n\#\Sigma$ by an 
Anosov automorphism of an infranilmanifold $M$ to obtain a higher codimension Anosov
diffeomorphism on $(\T^n\#\Sigma)\times M$. Then one can show using smoothing theory (in 
conjuction with~\cite{LR} and~\cite{FH})
that $(\T^n\#\Sigma)\times M$ is not diffeomorphic to any infranilmanifold if $\Sigma$ is an exotic 
sphere.\footnote{The authors would like to thank the referee for bringing this to our attention.} The advantage of our
construction is that it gives higher codimension Anosov diffeomorphisms on manifolds with
{\it irreducible} smooth structure; that is, on manifolds that are not diffeomorphic
to a smooth Cartesian product of two lower dimensional closed smooth manifolds. The following is a straightforward consequence of Theorem~\ref{thm_appndx} that we prove in the appendix.
\begin{prop}\label{prop0}
If $M$ is an $n$-dimensional closed oriented infranilmanifold and $\Sigma$ is a homotopy $n$-sphere such that $M\#\Sigma$ is not diffeomorphic to $M$ then $M\#\Sigma$ is irreducible.
\end{prop}

If $M$ is an $n$-dimensional ($n\neq 4$) nilmanifold and $\Sigma$ is not diffeomorphic to the
standard sphere $S^n$ then $M\# \Sigma$ is not diffeomorphic to $M$
\cite[Lemma 4]{FJ2}. For an infranilmanifold the situation is more
involved. We have the following.
\begin{prop}
\label{prop1} Let $M$ be an $n$-dimensional ($n\neq 4$) orientable
infranilmanifold with a $q$-sheeted cover $N$ which is a
nilmanifold. Let $\Sigma$ be an exotic homotopy sphere of order $d$
from the Kervaire-Milnor group $\Theta_n$. Then $M\#\Sigma$ is not
diffeomorphic to any infranilmanifold if $d$ does not divide $q$. In particular,
$M\#\Sigma$ is not diffeomorphic to $M$ if $d$ does
not divide $q$.
\end{prop}
\begin{proof}
We proceed via proof by contradiction. Assume that $M$ is
diffeomorphic to $M\#\Sigma$.
It was shown in~\cite{LR} that an isomorphism between the
fundamental groups of a pair of closed infranilmanifolds is always
induced by a diffeomorphism. Hence, by precomposing the assumed
diffeomorphism $M\to M\#\Sigma$ with an appropriate self
diffeomorphism of $M$, we obtain that $M$ and $M\#\Sigma$ are
diffeomorphic via a diffeomorphism inducing the identity isomorphism
of the fundamental group. (The fundamental groups of $M$ and
$M\#\Sigma$ are canonically identified.) By lifting this
diffeomorphism to the covering space $N\to M$, we see that $N$ and
$N\#q\Sigma$ are also diffeomorphic, and hence $q\Sigma$ is
diffeomorphic to $S^n$ because of~\cite[Lemma~4]{FJ2}. Therefore $d$
divides $q$ which is the contradiction proving that $M$ is not
diffeomorphic to $M\#\Sigma$. Also $M\#\Sigma$ is not diffeomorphic
to any other infranilmanifold by the result from~\cite{LR} cited
above.
\end{proof}

In the next section we provide brief background on Anosov
automorphisms and the Gromoll filtration of the group of homotopy spheres. Then we proceed with the
proof of Theorem~\ref{thm_main}. The last section is devoted to examples to
which Theorem~\ref{thm_main} applies. In particular, we establish the following
result.

\begin{prop}
\label{prop2_intro} Let $F$ be a finite group. Then for any
sufficiently large number $k$ there exists a flat Riemannian
manifold $M$ with holonomy group $F$ and a homeomorphic irreducible smooth
manifold $N$ such that
\begin{enumerate}
\item $N$ supports a codimension $k$ Anosov diffeomorphism,
\item $N$ is not diffeomorphic to any infranilmanifold; in particular, $N$ is not
diffeomorphic to $M$.
\end{enumerate}
\end{prop}

The authors are grateful to the referee for his/her very useful comments and
suggestions.

\section{Background}
\label{sec_backg}
\subsection{Anosov automorphisms and expanding endomorphisms}
\label{section_bck}
Let $G$ be a simply connected nilpotent Lie group equipped with a
right invariant Riemannian metric. Let $\widetilde{L}\colon G\to G$
be an automorphism of $G$ such that $D \widetilde{L}\colon \mathfrak
{g}\to \mathfrak {g}$ is hyperbolic, \ie the absolute values of its
eigenvalues are different from 1. Assume that there exists a
cocompact lattice $\Gamma\subset G$ preserved by $\widetilde{L}$, $\widetilde{L}(\Gamma)=\Gamma$.
Then $\widetilde{L}$ induces an {\it Anosov automorphism} $L$ of the
nilmanifold $M=G/\Gamma$.

If $\widetilde{E}\colon G\to G$ is an automorphism  such that the
eigenvalues of $D \widetilde{E}\colon \mathfrak {g}\to \mathfrak
{g}$ are all greater than 1 and $\widetilde{E}(\Gamma) \subset \Gamma$
then $\widetilde{E}$ induces a finite-to-one {\it expanding endomorphism} of the
nilmanifold $M$.

Existence of Anosov automorphism is a strong condition on $G$. Still
there are plenty of non-toral examples in dimensions six, eight and
higher (see \eg~\cite{DM} and references therein).

Infranilmanifolds are finite quotients of nilmanifolds obtained
through the following construction. Consider a finite group $F$ of
automorphisms of $G$. Then the semidirect product $G\rtimes F$ acts
on $G$ by affine transformations. Consider a torsion free cocompact
lattice $\Gamma$ in $G\rtimes F$. Let $M$ be the orbit space
$G/\Gamma$. This space is naturally a manifold since $\Gamma$ is
torsion-free. It is known that $\Gamma\cap G$ is a lattice in $G$
and has finite index in $\Gamma$. Hence $G/\Gamma\cap G$ is a
nilmanifold that finitely covers $M$.

A hyperbolic automorphism $\tilde L\colon G\to  G$ with $\tilde
L\circ\Gamma\circ\tilde L^{-1}=\Gamma$  induces an infranilmanifold {\it Anosov
automorphism} $L\colon M\to M$.

An expanding endomorphism $\tilde E\colon G\to G$ with $\tilde
E\circ\Gamma\circ \tilde E^{-1}\subset \Gamma$ induces an infranilmanifold {\it expanding
endomorphism} $E\colon M\to M$.
\begin{remark}
 With the definitions above it is clear that the covering
automorphisms $\tilde L$ and $\tilde E$, of $L$ and $E$ from
Theorem~\ref{thm_main}, commute as well. We will use this fact in the proof of Theorem~\ref{thm_main}.
\end{remark}
\begin{remark} \label{rmk22} More generally Anosov automorphisms (and expanding
endomorphisms) can be constructed starting from affine maps
$x\mapsto v\cdot\tilde L(x)$ for some fixed $v\in G$. In case of a
single automorphism of a nilmanifold this does not give anything new as one can change the group
structure of the universal cover by moving the identity element to
the fixed point of the affine map. (In the case of a single automorphism of an infranilmanifold
or a higher rank action, affine maps do give new examples as explained in~\cite{D} and~\cite{H} respectively.)
\end{remark}

\subsection{Gromoll filtration of Kervaire-Milnor group}
Recall that a {\it homotopy $n$-sphere} $\Sigma$ is a smooth
manifold which is homeomorphic to the standard $n$-sphere $S^n$. The
set of all oriented diffeomorphism classes of homotopy $n$-spheres
($n\ge 5$) is a finite abelian group $\Theta_n$ under the operation
\# of connected sum. This group was introduced and studied by
Kervaire and Milnor~\cite{KM}.

A simple way of constructing a homotopy $n$-sphere $\Sigma$ is to
take two copies of a closed disk $D^n$ and paste their boundaries
together by some orientation preserving diffeomorphism
$$
f\colon S^{n-1}\to S^{n-1}.
$$
This produces a homotopy sphere $\Sigma_f$. It is easy to see that
if $f$ is smoothly isotopic to $g$ then $\Sigma_f$ is diffeomorphic
to $\Sigma_g$. Therefore, the map $f \mapsto \Sigma_f$ factors
through to a map
$$
F\colon \pi_0(\textup{Diff}(S^{n-1}))\to \Theta_n,
$$
which is known, due to~\cite{C} and~\cite{Sm}, to be a group isomorphism for $n\ge6$.

View $S^{n-1}$ as the unit sphere in $\mathbb R^n$. Consider the
group $\textup{Diff}_k(S^{n-1})$ of orientation preserving
diffeomorphisms of $S^{n-1}$ that preserve first $k$ coordinates.
The image of this group in $\Theta_n$ is {\it Gromoll subgroup}
$\Gamma^n_{k+1}$. Gromoll subgroups form a filtration
$$
\Theta_n=\Gamma_1^n \supseteq \Gamma_2^n \supseteq \: \ldots \:
\supseteq \Gamma_n^n = 0.
$$

Cerf~\cite{C} has shown that $\Gamma^n_1 = \Gamma^n_2$ for $n\ge 6$.
Antonelli, Burghelea and Kahn~\cite{ABK} have shown that
$\Gamma_{2m-2}^{4m-1} \ne 0$ for $m\ge 4$ and that
$\Gamma_{2v(m)}^{4m+1} \ne 0$ for $m$ not of the form $2^l-1$, where
$v(m)$ denotes the maximal number of linearly independent vector
fields on the sphere $S^{2m+1}$. For some more non-vanishing results
and explicit lower bounds on the order of $\Gamma_k^n$, see
\cite{ABK}.

\subsection{Connected summing with $\Sigma\in \Gamma_{k+1}^n $}
Consider the group $\textup{Diff}(S^{n-1-k},B)$ of
diffeomorphisms of $S^{n-1-k}$ that are identity on an open ball
$B$. Note that any element of $\pi_k(\textup{Diff}(S^{n-1-k},B))$ is
represented by a diffeomorphism of $D^k\times D^{n-1-k}$, $D^{n-1-k}
\stackrel{\mathrm{def}}{=}S^{n-1-k} \backslash B$, which preserves the first $k$ coordinates and
is the identity near the boundary. The space of such diffeomorphisms
will be denoted by {$\textup{Diff}_k(D^k\times
D^{n-1-k},\partial)$}. There are natural inclusions
$$
\pi_k(\textup{Diff}(S^{n-1-k},B))\hookrightarrow
\pi_0(\textup{Diff}_k(S^{n-1}))\hookrightarrow
\pi_0(\textup{Diff}(S^{n-1})).
$$
Denote by $i$ the composition of these inclusions.

It is not very hard to show \cite[Lemma 1.13]{ABK2} that
$$
F\Big(i\big(\pi_k(\textup{Diff}(S^{n-1-k},B))\big)\Big)=\Gamma_{k+1}^n.
$$

Thus, given $\Sigma\in \Gamma_{k+1}^n $ and a manifold $M$ we can
realize $M\# \Sigma$ in the following way. Remove a disk $D^k\times
D^{n-k}$ from $M$. Consider the diffeomorphism $\Bar{g}\in
\textup{Diff}_k(D^k\times D^{n-1-k},\partial)$ which extends by
the identity map to $g\in \textup{Diff}_k(D^k\times \partial
D^{n-k},\partial)$ that represents $\Sigma$ in
$\pi_k(\textup{Diff}(S^{n-1-k},B))$.
Form the connected sum $M\# \Sigma$ by gluing $D^k\times D^{n-k}$ back
in using the identity map on $\partial(D^k)\times D^{n-k}$ and using $g$ on
$D^k\times \partial(D^{n-k})$.

\begin{remark}
It is important that $\Bar{g}$ is identity not only on the boundary
of $D^k\times D^{n-1-k}$ but also in a neighborhood of the boundary.
This is needed for the smooth structure to be well defined avoiding
the problem at the corners.
\end{remark}

\section{The proof of Theorem~\ref{thm_main}}
\subsection{The construction of the smooth structure}
Recall that $M=G/\Gamma$ is an infranilmanifold and $\Sigma\in \Gamma_{k+1}^n $. First we explain our model for $M\#\Sigma$ and
provide the basic construction of the diffeomorphism $f\colon
M\#\Sigma \to M\#\Sigma$. Then we explain how to modify $f$ to obtain
an Anosov diffeomorphism.

Start with an Anosov automorphism $L\colon M\to M$ with
$k$-dimensional stable distribution $E^s$ (in case the unstable
distribution is $k$-dimensional, consider $L^{-1}$ instead). Choose
coordinates in a small neighborhood $\U$ of a fixed point that comes
from the identity element $id$ in $G$ so that $L$ is given by the formula
$$
L(x,y)=(L_1(x),L_2(y)), \; (x,y)\in \U\cap L^{-1}\U
$$
where $x$ is $k$-dimensional, $L_1$ is contracting and $L_2$ is
expanding.

Next we choose a product of two disks $R_0^+=D_0^+ \times C_0$ in
the positive quadrant $\{(x,y)\colon (x,y)>0\}$ in the proximity of
the fixed point $(0,0)$. Here $D_0^+$ is $k$-dimensional and $C_0$
is $(n-k)$-dimensional. Consider
$$
R_0^-\stackrel{\mathrm{def}}{=}\{(-x,y)\colon (x,y)\in
R_0^+\}\stackrel{\mathrm{def}}{=}D_0^-\times C_0
$$
and $R_0\stackrel{\mathrm{def}}{=}D_0\times C_0$, where $D_0$ is the convex
hull of $D_0^+$ and $D_0^-$. Also let $\U_0^+$, $\U_0^-$ and $\U_0$,
$\U_0\supset \U_0^+\cap \U_0^-$, be small neighborhoods of $R_0^+$,
$R_0^-$ and $R_0$, respectively.

Define $R_i$, $R_i^+=D_i^+\times C_i$, $R_i^-=D_i^-\times C_i$,
$\U_i^+$, $\U_i^-$ and $\U_i$ as images of $R_0$, $R_0^+$, $R_0^-$,
$\U_0^+$, $\U_0^-$ and $\U_0$ under $L^i$, $i=1,2$, respectively. We
make our choices in such a way that $\U_0$, $\U_1$, and $\U_2$ are
disjoint.

Let $g\in \textup{Diff}_k(D_1^+\times \partial C_1, \partial)$ be a
diffeomorphism representing $\Sigma$. Glue $\Sigma$ in along the
boundary of $R_1^+$ using $g$ as described in the previous section.
Glue $-\Sigma$ in along the boundary of $R_1^-$ using $g^{-1}$
considered as a diffeomorphism in $\textup{Diff}_k(D_1^-\times
\partial C_1, \partial)$. To be more precise, we identify $R_1^+$ and
$R_1^-$ via a translation $t\colon R_1^+\to R_1^-$ (rather than a
reflection!) and glue $-\Sigma$ in using $t\circ g^{-1}\circ
t^{-1}$. Finally, glue $\Sigma$ in along the boundary of $R_2^+$
using $L\circ g\circ L^{-1}$. The resulting manifold $M\# 2\Sigma\#
-\Sigma$ is diffeomorphic to $M\# \Sigma$. Note that in the course
of the construction of $M\# 2\Sigma\# -\Sigma$ the leaves of the
unstable foliation undergo some cutting and pasting which results in
a smooth foliation that we call $W^u$. Locally the leaves of $W^u$
are given by the same formulae $x=const$. In contrast, the stable
foliation is being torn apart.


\begin{figure}[htbp]
\begin{center}

\begin{picture}(0,0)%
\includegraphics{exotic.pstex}%
\end{picture}%
\setlength{\unitlength}{3947sp}%
\begingroup\makeatletter\ifx\SetFigFont\undefined%
\gdef\SetFigFont#1#2#3#4#5{%
  \reset@font\fontsize{#1}{#2pt}%
  \fontfamily{#3}\fontseries{#4}\fontshape{#5}%
  \selectfont}%
\fi\endgroup%
\begin{picture}(6264,8485)(-731,766)
\put(1650,6300){\makebox(0,0)[lb]{\smash{{\SetFigFont{12}{14.4}{\rmdefault}{\mddefault}{\updefault}{\color[rgb]{0,0,0}$f(\tilde\V_0)$}%
}}}}
\put(5101,839){\makebox(0,0)[lb]{\smash{{\SetFigFont{12}{14.4}{\rmdefault}{\mddefault}{\updefault}{\color[rgb]{0,0,0}$x$}%
}}}}
\put(2176,8864){\makebox(0,0)[lb]{\smash{{\SetFigFont{12}{14.4}{\rmdefault}{\mddefault}{\updefault}{\color[rgb]{0,0,0}$y$}%
}}}}
\put(3826,839){\makebox(0,0)[lb]{\smash{{\SetFigFont{12}{14.4}{\rmdefault}{\mddefault}{\updefault}{\color[rgb]{0,0,0}$D_0^+$}%
}}}}
\put(451,839){\makebox(0,0)[lb]{\smash{{\SetFigFont{12}{14.4}{\rmdefault}{\mddefault}{\updefault}{\color[rgb]{0,0,0}$D_0^-$}%
}}}}
\put(2476,1664){\makebox(0,0)[lb]{\smash{{\SetFigFont{12}{14.4}{\rmdefault}{\mddefault}{\updefault}{\color[rgb]{0,0,0}$C_0$}%
}}}}
\put(226,1664){\makebox(0,0)[lb]{\smash{{\SetFigFont{12}{14.4}{\rmdefault}{\mddefault}{\updefault}{\color[rgb]{0,0,0}$R_0^-$}%
}}}}
\put(3601,1664){\makebox(0,0)[lb]{\smash{{\SetFigFont{12}{14.4}{\rmdefault}{\mddefault}{\updefault}{\color[rgb]{0,0,0}$R_0^+$}%
}}}}
\put(3901,3164){\makebox(0,0)[lb]{\smash{{\SetFigFont{12}{14.4}{\rmdefault}{\mddefault}{\updefault}{\color[rgb]{0,0,0}$\Sigma$}%
}}}}
\put(1351,3164){\makebox(0,0)[lb]{\smash{{\SetFigFont{12}{14.4}{\rmdefault}{\mddefault}{\updefault}{\color[rgb]{0,0,0}$-\Sigma$}%
}}}}
\put(3601,5414){\makebox(0,0)[lb]{\smash{{\SetFigFont{12}{14.4}{\rmdefault}{\mddefault}{\updefault}{\color[rgb]{0,0,0}$\Sigma$}%
}}}}
\put(2776,3689){\makebox(0,0)[lb]{\smash{{\SetFigFont{12}{14.4}{\rmdefault}{\mddefault}{\updefault}{\color[rgb]{0,0,0}$\tilde\U_1$}%
}}}}
\put(3100,3300){\makebox(0,0)[lb]{\smash{{\SetFigFont{12}{14.4}{\rmdefault}{\mddefault}{\updefault}{\color[rgb]{0,0,0}$\tilde\U_1^+$}%
}}}}
\put(550,3300){\makebox(0,0)[lb]{\smash{{\SetFigFont{12}{14.4}{\rmdefault}{\mddefault}{\updefault}{\color[rgb]{0,0,0}$\tilde\U_1^-$}%
}}}}
\put(2776,2039){\makebox(0,0)[lb]{\smash{{\SetFigFont{12}{14.4}{\rmdefault}{\mddefault}{\updefault}{\color[rgb]{0,0,0}$\tilde\U_0$}%
}}}}
\put(1825,4019){\makebox(0,0)[lb]{\smash{{\SetFigFont{12}{14.4}{\rmdefault}{\mddefault}{\updefault}{\color[rgb]{0,0,0}$\tilde\V_0$}%
}}}}
\end{picture}%

\end{center}
\caption{}\label{figure}
\end{figure}


\subsection{The construction of the diffeomorphism $f$}
Consider an open set $\V_0$ which is the union of $\U_1^-$, $\U_2^+$
and a small tube joining them as shown in Figure 1. Let
$\V_1=L(\V_0)$.

We proceed with definition of $f\colon M\# 2\Sigma\# -\Sigma\to M\#
2\Sigma\# -\Sigma$. From now on we will be using the same notation
for various regions in $M\# 2\Sigma\# -\Sigma$ as that for $M$ with
a tilde on top. For example, $\widetilde{\U}_1^+$ stands for
$\U_1^+$ with homotopy sphere $\Sigma$ glued in along the boundary of $R_1^+$.

Let $f(p)=L(p)$ unless $p$ belongs to $\widetilde{\U}_0$,
$\widetilde{\U}_1^+$ or $\widetilde{\V}_0$. We need to define $f$
differently on $\widetilde{\U}_0$, $\widetilde{\U}_1^+$, and
$\widetilde{\V}_0$ as the smooth structure on $L(\U_0)$, $\U_1^+$
and $\V_0$ has been changed.

The restriction $f|_{\widetilde{\U}_1^+}\colon {\widetilde{\U}_1^+}\to
{\widetilde{\U}_2^+}$ can be naturally induced by $L\colon
{{\U}_1^+}\to {{\U}_2^+}$ because of the way ${\widetilde{\U}_1^+}$
and ${\widetilde{\U}_2^+}$ were defined.

To define $f|_{\widetilde{\U}_0}\colon {\widetilde{\U}_0}\to
{\widetilde{\U}_1}$, we interpret $\widetilde{\U}_1$ as $\U_1$ with
$R_1$ being removed and then being glued back in via a
diffeomorphism $h$ that equals to $g$ along the boundary of $R_1^+$,
$g^{-1}$ along the boundary of $R_1^-$ and identity elsewhere.

The diffeomorphism $h$ is the concatenation of $g$ and $g^{-1}$, and
is easily seen to represent identity in
$\pi_k(\textup{Diff}(\partial C_1))$. Indeed one can explicitly
construct an isotopy from $h$ to $Id$ in $\textup{Diff}_k(R_1)$ by
joining the translation $t\colon R_1^+\to R_1^-$ (from the
definition of $g^{-1}\in \textup{Diff}_k(D_1^-, \partial C_1)$) to
$Id\colon R_1^+\to R_1^+$. It follows that $\U_1$ is diffeomorphic
to $\widetilde{\U}_1$ by a diffeomorphism $f_1$ that fixes the first $k$
coordinates and is equal to the identity map near $\partial \U_1=\partial
\widetilde{\U}_1$. Define
$$
f|_{\widetilde{\U}_0}\stackrel{\mathrm{def}}{=}f_1\circ L.
$$

Similar to above, $\V_0$ is diffeomorphic to $\widetilde{\V}_0$ by a
diffeomorphism $f_2$ that fixes the first $k$ coordinates and is
identity near $\partial \V_0=\partial \widetilde{\V}_0$. Define
$$
f|_{\widetilde{\V}_0}\stackrel{\mathrm{def}}{=}L\circ f_2^{-1}.
$$

It is clear that our definitions coincide on the boundary of the
regions so that diffeomorphism $f$ is well defined. Also notice that
$f$ preserves foliation $W^u$.

\subsection{Modifying $f$ to get an Anosov diffeomorphism on
$M\#\Sigma$} We fix a Riemannian metric on $M\# 2\Sigma\# -\Sigma$ in
the following way. On $(M\# 2\Sigma\# -\Sigma) \backslash
(\widetilde{\U}_1^+\cup \widetilde{\U}_1^-\cup \widetilde{\U}_2^+)$ we
use a Riemannian metric induced by a right invariant Riemannian metric on $G$ such
that $F\subset Iso(G)$, where $F$ is from~\ref{section_bck}. We
extend it to $\widetilde{\U}_1^+\cup \widetilde{\U}_1^-\cup
\widetilde{\U}_2^+$ in an arbitrary way.

Let $\W=\widetilde{\U}_0\cup \widetilde{\U}_1^+\cup
\widetilde{\V}_0$. The foliation $W^u$ is uniformly expanding everywhere
but in $\W$. (In fact, with a suitable choice of the Riemannian metric
$W^u$ is expanding on $\widetilde{\U}_1^+$ as well, but we won't use
this fact.) The number
$$
\alpha(f) \stackrel{\mathrm{def}}{=} \inf_{v\in T_xW^u, v\neq 0,\:
x\in\W}\Big(\frac{\|Df^3v\|}{\|v\|}\Big)
$$
measures maximal possible contraction along $W^u$ that occurs as a
point passes through $\W$. Therefore, if one can guarantee that the
first return time to $\W$
$$
N_1(f)\stackrel{\mathrm{def}}{=}\inf_{x\in\W}\{i\colon i\ge 3, f^i(x)\in \W\}
$$
is large compared to $\alpha(f)$, then $W^u$ will be expanding for $f$.


Next we will construct a cone field on $M\#2\Sigma\#-\Sigma$ that would give us the stable bundle.

Define $E^u_f\stackrel{\mathrm{def}}{=}TW^u\subset T(M\#2\Sigma\#-\Sigma)$. Recall that $E^s\subset TM$ is the stable bundle for $L$. Let
$\P\stackrel{\mathrm{def}}{=}\widetilde\V_0\cup\widetilde\U_1^+$. For any $x\in {(M\#2\Sigma\#-\Sigma)}
\backslash \P\;$ define $E^s(x)$ by identifying
${(M\#2\Sigma\#-\Sigma)}
\backslash \P$ and $M
\backslash ({\U}_1^+\cup{\V}_0)$. Fix a small $\varepsilon>0$ and define the cones
$$
\EuScript C(x)\stackrel{\mathrm{def}}{=}\{v\in T_xM:\measuredangle(v,E^s)<\varepsilon\}
$$
for $x\in (M\#2\Sigma\#-\Sigma)
\backslash \P$. Also define
\begin{multline*}
\quad\quad\quad\quad\quad\quad \quad\quad\quad\C(x)\stackrel{\mathrm{def}}{=}Df^{-1}(\C(f(x))),\; x\in\widetilde\V_0, \\
 \quad\quad\quad\quad\quad\quad\quad\quad\quad\shoveleft{ \C(x)\stackrel{\mathrm{def}}{=}Df^{-2}(\C(f^2(x))),\; x\in\widetilde\U_1^+. \hfill}
\end{multline*}
Note that $Df^{-1}(\C(x))\subset\C(f^{-1}(x))$ for $x\notin\P$ and $\C(x)\cap E_f^u(x)=\{0\}$ for all $x\in M\#2\Sigma\#-\Sigma$. Therefore for any
$x\in\P$ the sequence of cones $Df^{-i}(\C(x))$, $i>0$, shrinks exponentially fast towards $E^s$ until
the sequnce of base points $f^{-i}(x)$ enters $\P$ again. We see that if the first return time 
$$
N_2(f)\stackrel{\mathrm{def}}{=}\inf_{x\in\P}\{i\colon i\ge 2, f^{-i}(x)\in \P\}
$$
is large enough,
then there exists $N$ that depends on $\varepsilon$ and various choices we have made in our construction such that

\begin{enumerate}
 \item $\forall x\in \P$,\: $f^{-i}(x)\notin\P, i=2,\ldots N$,
\item $\forall x \in\P$, \:$Df^{-N}(\C(x))\subset \C(f^{-N}(x))$,
\item $\exists \lambda>1:\;\; \forall x\in \P$,\: $\forall v\in\C(x)$,\: $\|Df^{-N}v\|\ge\lambda\|v\|$.
\end{enumerate}
Now we modify the cone field in the following way.
\begin{multline*}
\quad\quad\quad \quad\quad\quad\bar\C(x)\stackrel{\mathrm{def}}{=}Df^{-i}(\C(f^i(x))),\;\;\;\mbox{if}\;\; x\in f^i(\P)\;\;\mbox{for}\;\; i=0,\ldots N-1, \\
 \quad\quad\quad\quad\quad\quad\shoveleft{\bar \C(x)\stackrel{\mathrm{def}}{=}\C(x),\;\;\;\mbox{if}\;\; x\notin\bigcup_{i=0}^{N-1}f^i(\P).\hfill}
\end{multline*}
It is clear from our definitions that the new cone field is invariant
$$
\forall x, \;\; Df^{-1}(\bar\C(f(x)))\subset\bar\C(x)
$$
and
$$
\exists\mu>1:\;\; \forall x \;\;\mbox{and}\;\;\forall v\in\bar\C(x),\;\;\|Df^{-1}v\|\ge\mu\|v\|.
$$
These properties imply~\cite[Section 6.2]{KH} that the bundle
$$
E^s_f(\cdot)\stackrel{\mathrm{def}}{=}\bigcap_{i\ge 0}Df^{-i}(\bar\C(f^i(\cdot)))
$$
is a $Df$-invariant $k$-dimensional exponentially contracting stable bundle for $f$.

We summarize that $f$ is Anosov provided that the first return times
$N_1(f)$ and $N_2(f)$ are large enough: how large depends only on the choices we
have made when constructing $M\# \Sigma$ and $f|_{\widetilde{\U}}$.

Recall that $M$
admits an expanding endomorphism $E$ that commutes with $L$. Let
$\widetilde{E}\colon G\to G$ be the covering automorphism of $E$.
The covering automorphism $\widetilde{L}$ of $L$ preserves the
lattice of affine transformations $\Gamma_m\stackrel{\mathrm{def}}{=}\widetilde{E}^m\circ\Gamma\circ\widetilde{E}^{-m}$, $m\ge 1$, \ie
$\widetilde{L}\circ\Gamma_m\circ\widetilde{L}^{-1}=\Gamma_m$. And
hence induces an Anosov automorphism $L_m$ of
$M_m\stackrel{\mathrm{def}}{=}G/\Gamma_m$.
Also we have the covering map $p\colon M_m\to M$ induced by $id\colon G\to G$ and the expanding
diffeomorphism $H\colon M\to M_m$ induced by $\widetilde{E}^m$.

We repeat our constructions of the exotic smooth structure and the
diffeomorphism on $M_m$ using the copy $\U_m$ of $\U$ in
$p^{-1}(\U)$ that contains
the $\Gamma_m$ orbit of $id$. Since $\U_m$ is
isometric to $\U$ and $L_m|_{\U_m}=L|_\U$ we can repeat the
constructions with the same choices as before and obtain a manifold
$M_m\# 2\Sigma\# -\Sigma$ together with a diffeomorphism $f_m$. In
particular, due to the same choices, $\alpha(f)=\alpha(f_m)$.

Let number $r_m$ be the maximal radius of a ball $B(id, r_m)\subset
G$ that projects injectively into $M_m$. Since $\widetilde{E}$ is
expanding $r_m\to \infty$ as $m\to \infty$. It follows that
$\min(N_1(f_m),N_2(f_m))\to \infty$ as $m\to \infty$. Hence, $f_m$ is Anosov for
large enough $m$. Since $M_m\# \Sigma$ is diffeomorphic to $M\#
\Sigma$ we obtain an Anosov diffeomorphism on $M\# \Sigma$ by
conjugating $f_m$ with this diffeomorphism.


\section{Examples}

Here we collect some examples where the conditions of Theorem~\ref{thm_main} are
satisfied. 
\subsection{Toral examples.} \label{exp_toral} Any Anosov
automorphism of the torus $\mathbb T^n$ commutes with the expanding endomorphism
$s\cdot Id$, $s>1$. Thus we obtain codimension $k$ Anosov
diffeomorphisms on $\T^n\#\Sigma$, where $\Sigma\in\Gamma^n_{k+1}$. If $\Sigma$ is 
not diffeomorphic to $S^n$ then, by the discussion in the
paragraph between~Propositions~\ref{prop0} and~\ref{prop1}, $\T^n\#\Sigma$ is exotic. And hence, 
by Proposition~\ref{prop0}, $\T^n\#\Sigma$ is irreducible. Therefore we obtain
codimension $k$ Anosov diffeomorphisms on exotic irreducible tori whenever
$\Gamma^n_{k+1}$ is non-trivial.
\subsection{A nilmanifold example.} Let $L_1\colon M_1\to M_1$ be
the Borel-Smale example of Anosov automorphism of a six dimensional
nilmanifold $M_1$ which is a quotient of the product of two copies
of Heisenberg group (see~\cite{KH}, section~4.17, for a detailed
construction). It is easy to check that
$$
\tilde E_1\colon \begin{pmatrix} 1 & x_1 & z_1 \\
0 & 1 & y_1\\
0 & 0 & 1
\end{pmatrix}\times
\begin{pmatrix} 1 & x_2 & z_2 \\
0 & 1 & y_2\\
0 & 0 & 1
\end{pmatrix}\mapsto
\begin{pmatrix} 1 & 2x_1 & 4z_1 \\
0 & 1 & 2y_1\\
0 & 0 & 1
\end{pmatrix}\times
\begin{pmatrix} 1 & 2x_2 & 4z_2 \\
0 & 1 & 2y_2\\
0 & 0 & 1
\end{pmatrix}
$$
induces an expanding endomorphism $E_1\colon M_1\to M_1$ that commutes with
$L_1$.

Also let $L_2\colon \mathbb T^{15}\to\mathbb T^{15}$ be a
codimension two Anosov automorphism and let $E_2\colon \mathbb
T^{15}\to\mathbb T^{15}$ be a conformal the expanding endomorphism as
in~\ref{exp_toral}.

We get that $L=L_1\times L_2$ is a codimension five Anosov
automorphism that commutes with the expanding endomorphism $E=E_1\times
E_2$. It is known that $|\Gamma_6^{21}|\ge 508$. Hence, Theorem~\ref{thm_main}
applies non-trivially to $L$ giving codimension five Anosov
diffeomorphisms on 21-dimensional exotic nilmanifolds
$(M_1\times\mathbb T^{15})\#\Sigma$, $\Sigma\in\Gamma_6^{21}$. 
The discussion at the end of Section~\ref{exp_toral} gives a reason why
these manifolds are exotic and irreducible.
\subsection{Infranilmanifold examples.}
Here the main goal is to show that Theorem~\ref{thm_main} can be used to
produce Anosov diffeomorphisms of high codimension on exotic
infratori. We use the term {\it ``infratorus"} as a synonym to ``closed
flat Riemannian manifold".

By a theorem of Auslander and Kuranishi~\cite{AK} there exists an
infratorus with holonomy group $F$ for any finite group $F$. There
are only finitely many groups of a given order $q$. Hence, there
exists a positive integer $k(q)$ such that for any finite group $F$
of order $q$ there is an infratorus with holonomy group $F$ of
dimension $k(q)$.
\begin{prop}
\label{prop2} Given a finite group $F$ of order $q$ and an integer
$k>k(q)$, there exists an (orientable) infratorus $M$ with holonomy
group $F$ and a homeomorphic irreducible smooth manifold $N$ such that
\begin{enumerate}
\item $N$ supports a codimension $k$ Anosov diffeomorphism,
\item $N$ is not diffeomorphic to any infranilmanifold; in particular, $N$ is not
diffeomorphic to $M$.
\end{enumerate}
\end{prop}

Clearly, Proposition~\ref{prop2} implies
Proposition~\ref{prop2_intro}.

\subsubsection{Existence of a commuting expanding endomorphism}
We start our proof of Proposition~\ref{prop2} by showing that the
assumption from Theorem~\ref{thm_main} about the existence of a commuting expanding
endomorphism is satisfied for some positive power of a given Anosov
automorphism.
\begin{prop}
\label{prop3} Let $M$ be an infratorus of dimension $n$ and $L\colon
M\to M$ be an Anosov automorphism. Then there exists an expanding
affine transformation (cf. Remarks~\ref{rmk22} and~\ref{rmk432}) that commutes with some positive power of $L$ and has a fixed
point in common with this power of $L$.
\end{prop}
\begin{proof}[The proof]
Let $\Gamma$ be the fundamental group of $M$. Then $\Gamma$ can be
identified with the group of deck transformations of $\R^n$. By
classical Bieberbach theorems we know that the intersection of
$\Gamma$ and the group of translations of $\R^n$ is a normal free
abelian group of finite index in $\Gamma$. We pick our basis for
$\R^n$ so that this group is $\Z^n\subset \R^n$. The finite quotient
group $F=\Gamma/\Z^n$ is the holonomy group of $M$. The holonomy
representation of $F$ into $GL_n(\R)$ is faithful and contained in
$GL_n(\Z)$. Thus we can identify $F$ with a subgroup of $GL_n(\Z)$.

We have the corresponding exact sequence
$$
0\longrightarrow \Z^n\longrightarrow\Gamma\longrightarrow
F\longrightarrow 1.
$$
Every element $\gamma\in\Gamma$ has the form $x\mapsto g_\gamma
x+u_\gamma$, where $g_\gamma\in F$, $u_\gamma\in\R^n$.

\begin{claim}\label{claim1}
Let $q$ be the order of $F$. There exists $u_0\in\R^n$ such that if
we denote by $t$ the translation $x\mapsto x+u_0$ then
$$
t\circ\Gamma\circ t^{-1}\subseteq \qqZ\rtimes GL_n(\Z).
$$
\end{claim}
\begin{proof}[The proof of Claim~\ref{claim1}]
Since $(u_\gamma\!\mod \Z^n)\in\T^n$ depends only on $g_\gamma$ the
function ${\gamma\mapsto (u_\gamma \mod \Z^n)}$ factors through to a
function
$$
\bar u\colon F\to \T^n.
$$
This function can be easily seen to be a crossed homomorphism, that
is, $\bar u(gh)=g\bar u(h)+\bar u(g)$ for all $g, h\in F$.

Let $\pi\colon\T^n\to\frac1q\T^n\stackrel{\mathrm{def}}{=}\R^n/\qZ$ be the
natural projection. The class of $\bar u$ in $H^1(F,\T^n)$ vanishes
after projecting to $H^1(F,\frac1q\T^n)$. This can be seen as
follows. Define
$$
u_0\stackrel{\mathrm{def}}{=}-\frac1q\sum_{g\in F}u_{\hat g},
$$
where $\hat g$ is any lift of $g$ to $\Gamma$. Then a direct
computation shows that
$$
\left(gu_0-u_0 \mod\qZ\right) = \pi(\bar u(g))
$$
for any $g\in F$. This means that $\pi\circ\bar u$ is a principal
crossed homomorphism. Now take any $\gamma\in\Gamma$,
$$
t\circ\Gamma\circ t^{-1}(x)=g_\gamma x-g_\gamma u_0+u_\gamma+u_0.
$$
But
$$\left(u_\gamma-g_\gamma u_0+u_0\mod\qZ\right)=\pi(\bar
u({g_\gamma}))-\left(g_\gamma u_0-u_0\mod\qZ\right)=0 $$ and the
claim follows.
\end{proof}

Therefore, by changing the origin of $\R^n$, we may assume from now
on that $\Gamma$ is a subgroup of $\qqZ\rtimes GL_n(\Z)$.

Fix an integer $s\equiv 1\mod q$, $s>1$. Using group cohomology
Epstein and Shub~\cite{ES} showed that there exists a monomorphism
$\varphi\colon\Gamma\to\Gamma$ that fits into the commutative
diagram
$$
\begin{CD}
0@>>>\Z^n @>>>\Gamma@>>> F@>>> 1\\
@.      @VVs\cdot IdV      @VV\varphi V           @VVId_FV   @.\\
 0@>>>\Z^n @>>>\Gamma@>>> F@>>> 1.
\end{CD}
$$
This implies that there exists an expanding map of $M$ that is
covered by an affine map of the form $x\mapsto s\cdot x+e_0$. Our
goal is to find an expanding map of $M$ that lifts to an origin
preserving expanding conformal map.

Define a function $\theta\colon F\to\qZ$ by
$$
\theta(g)\stackrel{\mathrm{def}}{=}s\cdot u_{\hat g}-u_{\varphi(\hat g)},
$$
where $\hat g$ is any lift of $g$ to $\Gamma$.

It is straightforward to check that $\theta$ is well defined and,
actually, a crossed homomorphism. It is also easy to see that
$\theta$, in fact, takes values in $\Z^n$. Indeed, recall that
$s\equiv 1\mod q$. Therefore
$$
\theta(g)\equiv u_{\hat g}-u_{\varphi(\hat g)}\equiv u_{\hat
g\circ\varphi(\hat g)^{-1}}\mod\Z^n.
$$
The latter expression is zero because $\hat g\circ\varphi(\hat
g)^{-1}$ is a pure translation.

Since the inclusion $\Z^n\hookrightarrow\qZ$ induces multiplication
by $q$ on cohomology and the abelian group $H^1(F,\qZ)$ has exponent
$q$~(see~\cite[p.85, 10.2]{Brown}), the composite
$$
\theta\colon F\to\Z^n\hookrightarrow\qZ
$$
is a principal crossed homomorphism; \ie there exists a vector
$v\in\qZ$ such that
$$\theta(g)=gv-v.$$
In fact, by a direct computation one can check that
$$
v=-\frac1q\sum_{g\in F}\theta(g).
$$

Now consider the conformal expanding affine transformation $E$ of
$\R^n$ defined by
$$
E(x)\stackrel{\mathrm{def}}{=}s\cdot x+v.
$$
\begin{claim}\label{claim2}
Conjugation by $E$ restricted to $\Gamma$ is $\varphi$; \ie
$$
\varphi(\gamma)=E\circ\gamma\circ E^{-1}
$$
for each $\gamma\in\Gamma$.
\end{claim}
\begin{proof}[The proof of Claim~\ref{claim2}]
We have the following sequence of equalities.
\begin{multline*}
E\circ\gamma\circ E^{-1}=(s\cdot x+v)\circ(g_\gamma
x+u_\gamma)\circ\left(\frac1s\cdot
x-\frac vs\right)=g_\gamma x-g_\gamma v+v+s\cdot u_\gamma\\
=g_\gamma x-\theta(g_\gamma)+s\cdot u_\gamma=g_\gamma x+s\cdot
u_\gamma-(s\cdot u_\gamma-u_{\varphi(\gamma)})=\varphi(\gamma).
\end{multline*}
\end{proof}
Consider the $q$-fold composition
$$
 E^q(x)=s^q\cdot x+\hat v, \;\;\;\mbox{where}\;\;\hat
v\stackrel{\mathrm{def}}{=}(s^{q-1}+s^{q-2}+\ldots +1)v.
$$
Denote by $\E$ the map of $M$ induced by $E^q$.
Since
$s^{q-1}+s^{q-2}+\ldots +1\equiv 0 \mod q$ we have that $\hat
v\in\Z^n$. Hence $\E(\mathbf o)=\mathbf o$, where $\mathbf o$
denotes the image of the new origin of $\R^n$ (after applying Claim~\ref{claim1})
under the covering projection $\R^n\to M$.

Now let $\tilde L\colon \R^n\to\R^n$ be a lift of $L$. It has the
form $x\mapsto Ax+x_0$ with respect to the new origin for $\mathbb
R^n$ (after applying Claim~\ref{claim1}), where $x_0\in\R^n$ and $A\in GL_n(\Z)$
is a hyperbolic matrix. Conjugation by $\tilde L$ induces an
automorphism $\psi\colon\Gamma\to\Gamma$ yielding the following
commutative diagram
$$
\begin{CD}
0@>>>\Z^n @>>>\Gamma@>>> F@>>> 1\\
@.      @VVAV      @VV\psi V           @VVV   @.\\
 0@>>>\Z^n @>>>\Gamma@>>> F@>>> 1.
\end{CD}
$$
Since $F$ is finite, by replacing $L$ and $\tilde L$ by some finite
powers, we can assume that the induced automorphism of $F$ is
identity. This implies that
$$
\forall\gamma\in\Gamma,\;\; A\circ g_\gamma=g_\gamma\circ
A.\eqno(\ast)
$$

Proceeding analogously to how we defined $\theta$ using $\varphi$,
we define $\omega\colon F\to\qZ$ by
$$
\omega(g)\stackrel{\mathrm{def}}{=}Au_{\hat g}-u_{\psi(\hat g)}.
$$
Again, it is straightforward to check that $\omega$ is well defined
and, due to~$(\ast)$, a crossed homomorphism.

Let $\bar A$ be the image of $A$ in the finite group $GL_n(\Z_q)$.
By replacing $L$ (and $\tilde L$ accordingly) by a further positive
power, we may also assume that $\bar A=Id$. We can check that the
image of $\omega$ is contained in $\Z^n$. Indeed, write
$$
Au_\gamma-u_{\psi(\gamma)}=(Au_\gamma-u_\gamma)+(u_\gamma-u_{\psi(\gamma)}).
$$
The first summand is in $\Z^n$ because $u_\gamma\in\qZ$ and $\bar
A=Id$. The second summand equals to
$u_{\gamma\circ\psi(\gamma)^{-1}}$ and belongs to $\Z^n$ since
$\gamma\circ\psi(\gamma)^{-1}$ is a pure translation.

Hence, just as for $\theta$, we get that
$$
\forall g\in F,\;\; \omega(g)=gw-w,
$$
where $w\in\qZ$.

Now consider the affine Anosov diffeomorphism $\A$ of $\R^n$ defined
by
$$
\A(x)\stackrel{\mathrm{def}}{=}Ax+w.
$$
\begin{claim}\label{claim3}
Conjugation by $\A$ restricted to $\Gamma$ is $\psi$, \ie
$$
\psi(\gamma)=\A\circ\gamma\circ\A^{-1}
$$
for each $\gamma\in\Gamma$.
\end{claim}
\begin{proof}[The proof of Claim~\ref{claim3}]
The claim follows from the computation below that uses the commutativity
property~$(\ast)$.
\begin{multline*}
\A\circ\gamma\circ\A^{-1}=(Ax+w)\circ(g_\gamma
x+u_\gamma)\circ(A^{-1}x-A^{-1}w)\\=A\circ g_\gamma\circ
A^{-1}x-A\circ g_\gamma\circ A^{-1}w+Au_\gamma+w =g_\gamma
x-g_\gamma w+w+Au_\gamma\\
=g_\gamma
x-\omega(g)+Au_\gamma=g_{\psi(\gamma)}x-(Au_\gamma-u_{\psi(\gamma)})+Au_\gamma=\psi(\gamma).
\end{multline*}
\end{proof}

Consider the $q$-fold composition
$$
\A^q(x)=A^q x+\hat w,\;\;\;\mbox{where}\;\; \hat
w\stackrel{\mathrm{def}}{=}A^{q-1}w+A^{q-2}w+\ldots +w.
$$
Denote by $\AA$ the Anosov diffeomorphism of $M$ induced by $\A^q$.
Since $\bar A=Id$ we have that $Au\equiv u\mod\Z^n$ for any
$u\in\qZ$. Thus $\hat w\equiv 0\mod\Z^n$; \ie $\hat w\in \Z^n$ and 
hence $\AA(\mathbf o)=\mathbf o$. (Recall that $\mathbf o\in M$ is
the image of $0\in\R^n$.) 

Therefore the conformal expanding endomorphism $\E\colon M\to M$ and
the Anosov diffeomorphism $\AA\colon M\to M$ have a common
fixed point $\mathbf o$, and hence commute.
\begin{claim}\label{claim4}
The diffeomorphism $\AA$ is affinely conjugate to a
positive power of $L$.
\end{claim}
\begin{proof}[The proof of Claim~\ref{claim4}]
It is clearly sufficient to find a vector $a\in\R^n$ such that the
translation $T\colon x\mapsto x+a$ has the following properties:
\begin{enumerate}
\item $T\circ\A^q\circ T^{-1}=\tilde L^q$,
\item $T\circ\gamma=\gamma\circ T$ for every $\gamma\in\Gamma$.
\end{enumerate}
Define vector $\hat u$ by the equation
$$
\tilde L^qx=A^qx+\hat u.
$$
Set $a=(Id-A^q)^{-1}(\hat u-\hat w)$. (Recall that $A$ has no
eigenvalues of length 1.) A straightforward calculation verifies the
first property. Another straightforward calculation shows that the
second property is equivalent to
$$
\forall \gamma\in\Gamma,\;\; g_\gamma a=a.
$$
Using~$(\ast)$ this is equivalent to
$$
\forall \gamma\in\Gamma,\;\; g_\gamma(\hat u-\hat w)=\hat u-\hat w.
$$
To verify this we use that
$$
\tilde L^q\circ\gamma=\psi^q(\gamma)\circ\tilde
L^q\;\;\;\mbox{and}\;\; \AA^q\circ\gamma=\psi^q(\gamma)\circ\AA^q
$$
for all $\gamma\in\Gamma$. Expanding these equations yields
$$
\hat u+A^qu_\gamma=g_{\psi^q(\gamma)}\hat
u+u_{\psi^q(\gamma)}\;\;\;\mbox{and}\;\; \hat
w+A^qu_\gamma=g_{\psi^q(\gamma)}\hat w+u_{\psi^q(\gamma)}
$$
for all $\gamma\in\Gamma$. And by subtracting the second equation
from the first yields
$$
\hat u-\hat w=g_{\psi^q(\gamma)}(\hat u-\hat w)=g_\gamma(\hat u-\hat
w).
$$
\end{proof}
Denote by $\TT$ the conjugacy between $\AA$ and $L^q$ induced by
$T$. To complete the proof we expand $\E\circ\AA=\AA\circ\E$ as
$$
\E\circ\TT^{-1}\circ L^q\circ \TT=\TT^{-1}\circ L^q\circ \TT\circ
\E,
$$
which means that $L^q$ commutes with the expanding endomorphism
$\TT\circ\E\circ\TT^{-1}$.
\end{proof}
\begin{remark}\label{rmk432}
In the course of the proof of Proposition~\ref{prop3} we have passed
to a positive finite power of $L$ twice (without changing the
notation for $L$) to guarantee that the induced homomorphism of $F$
is identity and $\bar A\in GL_n(\Z_q)$ is the identity matrix.
Therefore the actual power of the initial Anosov diffeomorphism that
commutes with the constructed expanding endomorphism may be larger than $q$.
Note also that if we change for a second time
the origin of $\mathbb R^n$ to a point lying over $\TT(\mathbf o)$ then $\TT\circ\E\circ\TT^{-1}$ and $L^q$ are an expanding
endomorphism and an Anosov automorphism, respectively. 
\end{remark}

\subsubsection{The construction of an Anosov automorphism of an infratorus with holonomy group $F$}

Here we present a construction of a codimension $k-1$ Anosov
automorphism. This is a generalization of a construction
in~\cite{Port}, where an equivalent construction for $s=2$ (see below) was carried out.

Let $m\stackrel{\mathrm{def}}{=}k-1$. Since $m\ge k(q)$ we can find an infratorus of
dimension $m$ whose holonomy group is $F$. Let $\Gamma$ be its
fundamental group. Then, as before, we have an exact sequence
$$
0\longrightarrow\Z^m \longrightarrow\Gamma\longrightarrow
F\longrightarrow 1.
$$

Fix an integer $s>1$ and let $A_1\in GL_s(\Z)$ be a matrix
representing a codimension one Anosov automorphism of $\T^s$.
See~\cite[Lemma~1.1]{FJ} where such a matrix $A_1$ is constructed.
Then the $m$-fold product
$$
A_m\stackrel{\mathrm{def}}{=}Id\otimes A_1\colon\Z^m\otimes\Z^s\to
\Z^m\otimes\Z^s
$$
is an Anosov matrix $A_m\in GL_{ms}(\Z)$ representing a codimension
$m$ Anosov automorphism of $\T^{ms}$.

The group $F$ acts faithfully on $\Z^m\otimes\Z^s$ by $g\otimes Id$ for
each $g\in F$ and $A_m$ obviously commutes with this action.

The matrix $A_m$ induces an automorphism
$$
(A_m)_*\colon H^2(F, \Z^m\otimes\Z^s)\to H^2(F, \Z^m\otimes\Z^s),
$$
which has a finite order since $H^2(F, \Z^m\otimes\Z^s)$ is a finite
group. Therefore some positive power of $(A_m)_*$ is identity.
Hence, after replacing $A_1$ by some positive power, we may assume that
$(A_m)_*=Id$.

Now consider an $s$-fold product
$\Gamma^s=\Gamma\times\Gamma\times\ldots \times\Gamma$. It fits into
an exact sequence
$$
0\longrightarrow(\Z^m)^s
\longrightarrow\Gamma^s\stackrel{p}{\longrightarrow}
F^s\longrightarrow 1.
$$
Let $\Gamma^{(s)}=p^{-1}(F)$, where $F$ is identified with the
diagonal subgroup of $F^s$.

Then we have a short exact sequence
$$
0\longrightarrow\Z^{ms} \longrightarrow\Gamma^{(s)}\longrightarrow
F\longrightarrow 1, \eqno(\star)
$$
where $\Gamma^{(s)}$ is torsion free and the diagonal action of $F$
on $(\Z^m)^s$ is faithful. Also note that $\Z^{ms}$ can be
identified with $\Z^m\otimes\Z^s$ in such a way that the diagonal
action of $F$ becomes the action described above on
$\Z^m\otimes\Z^s$. Recall that the extension~$(\star)$ determines an
element $\theta\in H^2(F,\Z^k)$. And, since
$(A_m)_*(\theta)=\theta$, there is an automorphism
$A\colon\Gamma^{(s)}\to\Gamma^{(s)}$ such that the following diagram
commutes

$$
\begin{CD}
0@>>>\Z^m\otimes\Z^s @>>>\Gamma^{(s)}@>>> F@>>> 1\\
@.      @VVA_mV      @VVAV           @VVId_FV   @.\\
 0@>>>\Z^m\otimes\Z^s @>>>\Gamma^{(s)}@>>> F@>>> 1
\end{CD}
$$
(see p.~94 of~\cite{Brown}).

By the Bieberbach theorems $\Gamma^{(s)}=\pi_1(M_F)$, where $M_F$ is
a $ms$-dimensional infratorus with holonomy group $F$. Furthermore,
$A$ is induced by a codimension $m$ affine Anosov diffeomorphism of
$M_F$ which we denote by $A$ as well. After changing, if needed, the 
origin $0$ of the affine space $\R^n$, $A(0)=0$. Hence $A$ induces
an Anosov automorphism of $M_F$.

Note that if $s$ is even then $\Gamma^{(s)}$ acts by orientation
preserving transformations. Therefore $M_F$ is orientable when $s$
is even.

\subsubsection{The proof of Proposition~\ref{prop2}}

It follows from the proof of Theorem~$3'''$ of \cite{FO} that given
$k, u\in\Z^+$ there exists $d(k,u)$ such that for all integers $d>d(k,u)$
there is an element in the Gromoll group $\Gamma_{k+1}^{4d+3}$ of an
odd order larger than $u$.

Take $u=q$ and apply construction of the previous subsection with
$sm=s(k-1)>4d(k,u)$ and $s$ even. This gives a codimension $k-1$
Anosov automorphism $A\colon M_F\to M_F$. Let $L_1\colon\mathbb
T^\sigma\to\mathbb T^\sigma$ be a codimension 1 Anosov automorphism.
Choose $\sigma\in\{2,3,4,5\}$ so that $s(k-1)+\sigma\equiv 3\mod 4$.

Let $M\stackrel{\mathrm{def}}{=}M_F\times \mathbb T^\sigma$ and $L=A\times L_1$. Then $L$ is
a codimension $k$ Anosov automorphism of the orientable infratorus
$M$ with holonomy group $F$. By the construction, $\Gamma_{k+1}^{\dim
M}$ has an element $\Sigma$ of order larger than $q$. 

Let $N\stackrel{\mathrm{def}}{=}M\#\Sigma$, then $N$ is not diffeomorphic to $M$ as well as
to any other infranilmanifold by Proposition~\ref{prop1}. By
Theorem~\ref{thm_main} and Proposition~\ref{prop3}, $N$ supports a codimension $k$
Anosov diffeomorphism; cf.~Remark~\ref{rmk432}. Finally,
Proposition~\ref{prop0} yields that $N$ is irreducible.

\section{Appendix: Irreducible smooth structures}

This additional section is devoted to proving the following result.

\begin{theorem}
\label{thm_appndx}
 Let $M$ be an $n$-dimensional closed oriented infranilmanifold, $n\ge 7$,
and $\Sigma$ be an exotic $n$-sphere. If $M\#\Sigma=N_1\times N_2$, \ie is a smooth
Cartesian product, where $\dim N_i\ge 1$, $i=1,2$, then $M\#\Sigma$ is diffeomorphic to $M$. 
\end{theorem}

Before proving this result, we need some pertinent facts from smoothing theory which can be found in~\cite{KS}.

Let $Y$ be a smooth $n$-dimensional manifold, where $n\ge 5$. A {\it smooth structure}
on $Y$ is a homeomorphism $\varphi\colon X\to Y$, where $X$ is a smooth manifold. Two such structures
$$
\varphi_i\colon X_i\to Y, i=1,2
$$
are {\it concordant}, if there exists a smooth manifold $W$ and a homeomorphism
$$
\Phi\colon W\to Y\times[0,1],
$$
such that $\partial W=X_1\sqcup X_2$ and $\Phi|_{X_i}=\varphi_i$.

Let $[\varphi]$ denote the concordance class of $\varphi$ and $\EuScript S(Y)$ be the set of all such classes.
Then $\EuScript S(Y)$ is in natural bijective correspondence with $[Y, Top/O]$, where $Top/O$ is an infinite loop
space and $[Y, Top/O]$ denotes the set of all homotopy classes of continuous maps from $Y$ to $Top/O$.
Note that $[Y, Top/O]$ is an abelian group. In this way, $\EuScript S(Y)$ acquires an abelian group structure.
Given a smooth structure $\varphi\colon X\to Y$, we let  $\widehat\varphi\colon Y\to Top/O$ denote a representative 
of the corresponding homotopy class. Here are some properties of this correspondence.
\begin{enumerate}
 \item $\widehat{Id}_Y$ is homotopic to a constant map and $[Id_Y]=0$.
\item Let $\sigma\colon\Sigma\to S^n$ be an exotic sphere and $\sigma_Y\colon Y\#\Sigma\to Y$ be the usual 
homeomorphism then $[\widehat{\sigma}_Y]=[\widehat\sigma\circ f_Y]$, where $f_Y\colon Y\to S^n$ is a degree one map.
\item  Let $\alpha\colon U\to Y$ denote the inclusion of an open subset $U\subset Y$, then $[\widehat\varphi\circ{\alpha}]=[\widehat\varphi_U]$, where $\varphi_U\colon\varphi^{-1}(U)\to U$ is the restriction
of $\varphi$ to $\varphi^{-1}(U)$.
\item {\bf Product Structure Theorem.} The homeomorphism $\varphi\times Id_{\mathbb R^m}\colon X\times\R^m\to
Y\times\R^m$ is a smooth structure on $Y\times\R^m$, and the map $[\varphi]\to[\varphi\times Id_{\mathbb R^m}]$
is a bijection of smooth structure sets $\EuScript S(Y)\to\EuScript S(Y\times\R^m)$. Furthermore,
$[\widehat{\varphi\times Id_{\mathbb R^m}}]=p^*[\widehat\varphi]$, where $p\colon Y\times\R^m\to Y$ denotes 
projection onto the first factor.
\end{enumerate}

We now start the proof of Theorem~\ref{thm_appndx}. Since $\pi_1(M)=\pi_1(N_1)\times\pi_1(N_2)$ and 
$\pi_1(M)$ is torsion-free, finitely generated and
virtually nilpotent group, it follows that $\pi_1(N_1)$ and $\pi_1(N_2)$ are also torsion-free, finitely generated and virtually nilpotent groups. By Mal$'$cev's work~\cite{M} (cf.~\cite[p.231]{W}) any such group is the fundamental group of a closed infranilmanifold. Hence there exist closed infranilmanifolds
$M_1$ and $M_2$ with $\pi_1(M_i)=\pi_1(N_i)$, $i=1,2$. Note that $N_i$, $i=1,2$, are asphericial. 
Therefore, using~\cite[Theorem~6.3]{FH}, we obtain homeomorphisms $f_i\colon
N_i\to M_i$. (Theorem~5.1 of~\cite{FH} (of which Theorem~6.3 is a corollary) was extended to dimension 4  by~\cite[Section~11.5]{FQ} and follows from results of Perelman 
(see \eg~\cite{CZ}) in dimension~3.) Note that $M$ and $M_1\times M_2$ are both closed infranilmanifolds with a specified 
isomorphism between their fundamental groups. Now the smooth rigidity result of  Lee and 
 Raymond~\cite{LR} yields a diffeomorphism
$M_1\times M_2\to M$ which makes the following diagram commute up to homotopy:
$$
\begin{CD}
 M\#\Sigma @>\sigma_M>> M\\
@|@AAA\\
N_1\times N_2 @>>{f_1\times f_2}> M_1\times M_2
\end{CD}
$$
We also orient $N_i$ and $M_i$, $i=1,2$, so that all pertinent maps
are orientation preserving and identify
$M_1\times M_2$ with $M$ by the vertical diffeomorphism
in the diagram. Using~\cite{FH} again, we see that $\sigma_M$ is concordant
to $f_1\times f_2$; that is,
$$
[\sigma_M]=[f_1\times f_2] \;\;\;\;\mbox{in}\;\;\EuScript S(M).
$$

We now complete the proof under the additional assumption that $\dim N_i\ge 5$ for $i=1$ and $i=2$.
After doing this, we will indicate the modifications needed to prove  Theorem~\ref{thm_appndx}  when this assumption
is dropped.

Identify $\R^s$ with an open ball in $M_2$, where $s=\dim M_2$. We intend to apply property 3 to the
inclusion 
$$
\alpha\colon U\stackrel{\mathrm{def}}{=}M_1\times\R^s\to M_1\times M_2=M
$$
and the smooth structure
$$
\varphi\stackrel{\mathrm{def}}{=}f_1\times f_2\colon N_1\times N_2\to M_1\times M_2.
$$
Notice, in this situation, that
$$
\varphi_U\colon\varphi^{-1}(U)\to U
$$
is the same as
$$
f_1\times(f_2|_V)\colon N_1\times V\to M_1\times \R^s,
$$
where $V=f_2^{-1}(\R^s)$. Moreover the Product Structure Theorem (property~4) yields that
$$
[f_1]=0\;\;\;\mbox{if and only if}\;\;\; [f_1\times f_2|_V]=0.
$$
To see that $[f_1\times f_2|_V]=0$, we recall that $f_1\times f_2$ is concordant to $\sigma_M$. This fact together with property~2 yields that
$$
[\widehat{f_1\times f_2}]=[\widehat{\sigma}_M]=[\widehat\sigma\circ f_M]\;\;\;\mbox{in}\;\;\; [M, Top/O].
$$
Therefore property~3 yields that
$$
[\widehat{f_1\times (f_2|_V)}]=[\widehat{f_1\times f_2}\circ\alpha]=[\widehat\sigma\circ f_M\circ\alpha]\;\;\;\mbox{in}\;\;\; [M_1\times V, Top/O].
$$
But $f_M\circ \alpha\colon M_1\times V\to S^n$ is homotopic to a constant map since $V$ is homotopic to a point and $\dim M_1<n$. Therefore $[f_1\times(f_2|_V)]=0$ and consequently $[f_1]=0$. Property~1 implies that $f_1$ is homotopic to a 
diffeomorphism $\bar f_1\colon N_1\to M_1$. And a completely analogous argument shows that $f_2$ is also homotopic to a 
diffeomorphism $\bar f_2\colon N_2\to M_2$. Consequently, $N_1\times N_2=M\#\Sigma$ is diffeomorphic to
$M_1\times M_2=M$ which is the posited result.

We finish by briefly indicating how to modify the above argument to complete the proof in general; \ie after dropping 
the assumption that $\dim N_1\ge 5$, $i=1,2$.

For this purpose, consider the smooth structure
$$
(M\#\Sigma)\times\R^{10}\stackrel{\sigma_M\times Id_{\mathbb R^{10}}}{\xrightarrow{\hspace*{1cm}}} M\times\R^{10}.
$$
Because of the Product Structure Theorem, it suffices to show that
$[\sigma_M\times Id_{\mathbb R^{10}}]=0$ in $\EuScript S(M\times\R^{10})$. This is accomplished
by showing that
$$
[f_i\times Id_{\mathbb R^5}]=0 \;\;\;\;\mbox{in}\;\;\;\EuScript S(M_i\times\R^5),\;i=1,2,
$$
which is proven by an argument similar to the one given above which verified that
$[f_i]=0$ in $\EuScript S(M_i)$ when $\dim M_i\ge 5$, $i=1,2$.

\end{document}